\documentclass[11pt]{article}
\usepackage{amsbsy}
\usepackage{amssymb,amsthm, amsmath, latexsym}
\usepackage{mathrsfs}
\usepackage[english]{babel}
\usepackage{polski}
\usepackage{color}
\usepackage[
    colorlinks=true,
    linkcolor= blue,
    citecolor=red
]{hyperref}
\usepackage{amsmath}
\usepackage{mathrsfs}
\usepackage{graphicx,amsmath,mathtools}
\usepackage{mathrsfs}
\usepackage{tikz}
\usepackage{wrapfig}
\usepackage{mathtools}
\usepackage{esint}
 \usepackage{stmaryrd}
\usepackage{amssymb,url,xspace}
 \usepackage[mathlines]{lineno}

\usepackage[breakable]{tcolorbox}
\usepackage{xcolor}
\usepackage{soulutf8}
\soulregister\cite7
\soulregister\underline7
\soulregister\eqref7
\soulregister\ref7
\definecolor{mycolor}{RGB}{229,228,200}
\definecolor{pinkcolor}{RGB}{247,202,193}
\definecolor{yel}{RGB}{255,255,179}

\newtcolorbox{warningbox}{
    colback=pinkcolor,    
    colframe=pinkcolor,   
    boxrule=0pt,
    sharp corners,      
    left=0pt,
    right=0pt,
    top=0pt,
    bottom=0pt,
    breakable}

\newtcolorbox{changebox}{
    colback=mycolor,    
    colframe=mycolor,   
    boxrule=0pt,
    sharp corners,      
    left=0pt,
    right=0pt,
    top=0pt,
    bottom=0pt,
    breakable}

\newtheorem{theorem}{Theorem}[section]
\newtheorem{proposition}{Proposition}[section]

\newtheorem{lemma}{Lemma}[section]

\newtheorem{definition}{Definition}[section]
\newtheorem{remark}{Remark}[section]

\numberwithin{equation}{section} \numberwithin{theorem}{section}
\numberwithin{proposition}{section} \numberwithin{lemma}{section}
\numberwithin{corollary}{section}
\numberwithin{definition}{section} \numberwithin{remark}{section}

 \newcommand{\T}{\S\kern .15em\relax }
 
\newcommand{\dm}{\,\mathrm{d}} 

\newcommand{\R}{\mathbb{R}}

\newcommand{\C}{{\mathcal C}}

\newcommand{\diam}{{\rm diam}}

\newcommand{\Div}{\mathrm{div} \ }

\newcommand{\Haus}{\mathcal{H}^{2}}

\newcommand{\Ccal}{\mathcal{C}}

\newcommand{\eps}{\varepsilon}
\newcommand{\curve}{\mathrm{Hom}}
\newcommand{\Lip}{\mathrm{Lip}}
\newcommand{\mres}{\mathbin{\vrule height 1.6ex depth 0pt width 0.13ex\vrule height 0.13ex depth 0pt width 1.3ex}}
\newcommand{\dist}{\mathrm{dist}}

\newcommand{\surf}{S}

\title{Existence and Regularity of Minimizers for a Plateau Approximation Problem}

\author{{\sc Eve Machefert}\footnote{INSA Lyon, CNRS, Ecole Centrale de Lyon, Université Claude Bernard Lyon 1, Université Jean Monnet, Intitut Camille Jordan, UMR5208, 69621 Villeurbanne, France ({\tt eve.machefert@insa-lyon.fr}).}}

\date{\today}

\begin{document}


\maketitle

\begin{abstract} In this paper, we study the functional introduced by the author in collaboration with Bonnivard, Bretin, and Lemenant \cite{bonnivard2025phasefieldapproximationplateaus}, which is designed to approximate Plateau’s problem. We establish the existence of a minimizer and prove its Hölder regularity. Our results may be viewed as a generalization to higher-dimensional surfaces of the one-dimensional work of Bonnivard, Lemenant, and Millot \cite{bonnivard2018phase} on the approximation of the Steiner problem.
\end{abstract}

\tableofcontents

\section{Introduction}

This paper is devoted to the study of the functional introduced in \cite{bonnivard2025phasefieldapproximationplateaus} to approximate a Plateau problem, which consists in finding a surface of minimal area spanning a collection of closed curves $\gamma_0,..., \gamma_n$ contained in the boundary of an open, bounded, convex set $\Ccal_0 \subset \R^3$. 

More precisely, the competitor surfaces are defined as the images of homotopies connecting the given curves. We define below the set of admissible homotopies connecting a curve $\gamma_i$ to $\gamma_j$:
\[
\curve(\gamma_i,\gamma_j) : = \{ \ell \in \mathrm{Lip}([0,1] \times \mathbb{S}^{1}, \overline{\Ccal_0}) \text{ such that } \ell(0) = \gamma_i \text{ and } \ell(1) = \gamma_{j} \}.
\]
Note that we require more than mere homotopies: we specifically consider Lipschitz homotopies. For any homotopy $\ell \in \curve(\gamma_i,\gamma_j)$, we define the associated surface $\surf_\ell$ as its image:
\begin{equation}\label{Def:surface_image}
	\surf_{\ell} := \ell([0,1] \times \mathbb{S}^{1}) \subset \overline{\Ccal_0}.
\end{equation}
We emphasize that the surface $\surf_{\ell}$ is $\mathcal{H}^{2}$-rectifiable and has finite $\mathcal{H}^{2}$-measure. Moreover, for any $u\in H^1(\R^3)$, the trace of $u$ on $\surf_{\ell}$ is well defined, since the set of points that are not Lebesgue points of $u$ is $\mathcal{H}^2$-negligible. 
However, to obtain the existence result, we need to impose additional properties on the competitors.
For $\Lambda>0$, we consider
\[
\curve^\Lambda(\gamma_i,\gamma_j) : = \{ \ell \in \curve(\gamma_i,\gamma_j) \text{ such that }  \;\Lip(\ell)\leq \Lambda \text{ and } \surf_\ell \text{ is } \Lambda-\text{upper Ahlfors regular} \ \}.\]

Let us recall the definition of upper Ahlfors regularity.
\begin{definition}
Let $E\subset \R^3$ and $\Lambda>0$. The set $E$ is said to be $\Lambda$-upper Ahlfors regular if for all $x\in E$ and all $r>0$ we have 
\[\Haus(E\cap B(x,r)) \leqslant \Lambda \pi r^2.\]
\end{definition}

The definition of our functional is based on the notion of a generalized geodesic distance between curves, associated with a given weight function $u$, which is defined as follows.
\begin{definition}\label{Def:geoDistance}
	Let $(\delta_\eps)_{\eps>0}$ be a sequence of positive numbers and $u \in C(\overline{\Ccal_0})$. We define the geodesic distance between $\gamma_i$ and $\gamma_j$:
	\begin{align*}
			&d_{u}(\gamma_i,\gamma_j) := \inf\left \{  \int_{\surf_{\ell}}( u^2 + \delta_\eps) d\Haus \ |\ \ell \in \curve^\Lambda(\gamma_i,\gamma_j)\right \}.
	\end{align*}
\end{definition}
Note that, in \cite{bonnivard2025phasefieldapproximationplateaus}, the geodesic distance is defined for Lipschitz homotopies $\ell$. The assumptions of uniform Lipschitz regularity and uniform upper Ahlfors regularity of the images are not required there. However, in the present work, we need to impose these assumptions in order to obtain the compactness result necessary for the existence theorem.

Let $\Ccal$ an smooth open bounded convex set containing the closure $\overline{\Ccal_0}$. This set $\Ccal$ will serve as our domain of study. Indeed, since we rely on PDE techniques, it is convenient to work within a smooth bounded domain. We now recall the definition of the functional introduced in \cite{bonnivard2025phasefieldapproximationplateaus}:
	\begin{equation}
		\label{functionalGeneral}
		F_{\eps}(u) := \eps \int_{\mathcal{C}}|\nabla u|^{2} dx + \frac{1}{4\eps}\int_{\mathcal{C}}(1-u)^{2}dx + \frac{1}{c_{\eps}}\sum_{i=1}^{n}d_{u}(\gamma_0, \gamma_{i}).
	\end{equation}
When, the sequences of positive numbers $(\delta_\eps)$ and $(c_\eps)$ are assumed to converge to zero, and to satisfy that $\delta_\eps/c_\eps$ converges to zero as $\eps \to 0$, \cite{bonnivard2025phasefieldapproximationplateaus} establishes that this functional approximate some Plateau's problem through a $\Gamma$-convergence result. 

For simplicity, we will assume in the following that the prescribed boundary consists of only two curves, $\gamma_0$ and $\gamma_1$. However, the results established in the present article remain valid when the boundary contains more than two curves.

In this paper, we are specifically interested in the decoupled functional defined below.
\begin{definition}
	For $u \in H^{1}(\C)$ such that $0\leqslant u \leqslant 1$, with $u = 1$ on $\partial{\C}$, and $\ell \in \curve^\Lambda(\gamma_0,\gamma_1)$,
	\begin{align*}
		&E_{\varepsilon}(u,\ell) := \varepsilon\int_{\mathcal{C}} |\nabla u|^{2} dx + \frac{1}{4\varepsilon}\int_{\mathcal{C}}(1-u)^{2} dx + \frac{1}{c_{\varepsilon}}\int_{\surf_{\ell}} (u^2 + \delta_\eps) d\Haus.
	\end{align*}
\end{definition}

\begin{remark}
	From Definition~\ref{Def:geoDistance} of the geodesic distance between closed curves, we know that for all $u\in H^{1}(\Ccal)$ such that $0\leqslant u \leqslant 1$ and $u = 1$ on $\partial{\Ccal}$,  
		\[F_{\varepsilon}(u) = \inf\left \{ E_{\varepsilon}(u,\ell)\ ,\ell \in \curve^\Lambda(\gamma_0,\gamma_1)  \right \}.\]
\end{remark}

The main result of this paper is the existence of a pair that minimizes this decoupled energy $E_\varepsilon$
\begin{equation*}
    \inf_{(u,\ell)} E_\varepsilon(u,\ell).
\end{equation*}

\begin{theorem}
\label{main th}

There exists a minimizing couple $(u,\ell)$ of the following problem in both variables
\begin{equation}
    \label{pb min}
    \min\left \{E_\eps(u,\ell), \, u \in  H^{1}(\C) \text{ such that } 0\leqslant u \leqslant 1 , \, u = 1 \text{ on } \partial{\C} \text{ and } \ell \in \curve^\Lambda(\gamma_0,\gamma_1) \right\}.
\end{equation}

Moreover, the minimizer $u$ is $\alpha$-Hölder continuous for all $0<\alpha<1$ with the estimate:
\[\|u\|_{C^{0,\alpha}(\mathcal{C})} \leqslant C_\alpha \frac{1+\Lambda c_\varepsilon ^{-1}}{\varepsilon^{\alpha}}.\]
	
\end{theorem}

The study of this energy is motivated by the $\Gamma$-convergence result established in \cite{bonnivard2025phasefieldapproximationplateaus}, where this decoupled functional was introduced both in the proof of the $\Gamma$-convergence and for numerical applications. Recall that the phase-field approach they proposed is a generalization to Plateau’s problem of the method introduced in \cite{bonnivard2015approximation} for Steiner’s problem. A similar study of the existence of solutions for the decoupled functional in the Steiner case was carried out in \cite{bonnivard2018phase}. The present work is inspired by this paper, but there are some differences that we explain below.

To prove Theorem~\ref{main th}, we first consider the optimization problems with either $u$ or $\ell$ fixed. As usual, establishing the existence of minimizers requires both a compactness result and the lower semi-continuity of the functional. The compactness follows from the choice of the class of competitors, so the main issue is to establish lower semi-continuity.

We begin by studying the following minimization problem 
\begin{equation}
\label{pb min l fixe}
    \inf_{u} E_\varepsilon(u,\ell), \text{ where } \ell \in \curve^\Lambda(\gamma_0,\gamma_1).
\end{equation}
The following Proposition will be proven in Section~\ref{section l fixe}. 
\begin{proposition}
For any given $\ell \in \curve^\Lambda(\gamma_0, \gamma_1)$, there exists a unique minimizer $u \in H^1(\Ccal)$ of $ E_\varepsilon(\cdot,l)$ which is globally $C^{0,\alpha}$, for all $0< \alpha <1$: there exists a constant $C_\alpha>0$ such that 
\[\|u \|_{C^{0,\alpha}(\mathcal{C})} \leqslant C_\alpha \frac{1+\Lambda c_\varepsilon ^{-1}}{\varepsilon^{\alpha}}.\]
\end{proposition}
The assumption that the surface $S_\ell$ is $\Lambda$-upper Ahlfors regular is the key ingredient to establish lower semi-continuity, and hence the existence of minimizers in this setting. We then apply classical results from elliptic regularity to derive Hölder estimates for the minimizer. Note that these estimates depend explicitly on $\Lambda$ and $\varepsilon$, but are independent of the specific surface $S_\ell$. This part is strongly inspired by \cite{bonnivard2018phase}.

Then, in Section~\ref{section u fixe} we consider the minimization problem
\begin{equation}
\label{pb min u fixe}
    \inf_{\ell \in \curve^\Lambda} E_\varepsilon(u,\ell), \text{ where } u \in H^1(\Ccal)\cap C(\overline{\Ccal}),
\end{equation}
and prove the following Proposition.
\begin{proposition}
\label{prop 1.2}
For all $u \in H^1(\C)\cap C(\overline{\Ccal})$ there exists a minimizer for the problem \eqref{pb min u fixe}. 
\end{proposition}

As previously mentionned, the main difficulty in the proof of Proposition~\ref{prop 1.2} lies in establishing the lower semicontinuity of the functional, which relies on the following Lemma.

 \begin{lemma}
 \label{golab}
Let $\ell_n : [0,1] \times  \mathbb{S}^1 \to \R^3$ be a sequence of uniformly Lipschitz functions, which converges uniformly to a Lipschitz function $\ell  : [0,1] \times  \mathbb{S}^1 \to \R^3$.
Then, for all open set $A$

$$\liminf \Haus \mres S_{\ell_n}(A) \geqslant \Haus \mres S_\ell(A),$$
where $S_\ell$ (respectively $S_{\ell_n}$) denotes the image of the Lipschitz function $\ell$ (respectively $\ell_n)$.
\end{lemma}

This Lemma, proven in Section~\ref{section u fixe}, constitutes a central and original contribution of the present article and relies crucially on the specific definition of the class of surfaces under consideration. It can be viewed as a generalization of Go\l{}\k{a}b’s theorem to the two-dimensional Hausdorff measure, and its proof follows a similar strategy. In Go\l{}\k{a}b’s theorem, the first step consists in establishing the rectifiability of the limit set; in our setting, this property is immediate, since the surfaces are assumed to be images of Lipschitz functions. We then study the weak limit of the measures $(\Haus \mres S_{\ell_n})$ and, using the tangent plane of the limit surface, derive a lower bound on the upper density of the limit measure to conclude. The key point of our proof relies on a topological argument, stating that a Lipschitz image, which is uniformly close to a disk, must have at least almost the density of a disk (see the proof of Lemma~\ref{golab}).

We emphasize that the proof of this Lemma does not rely on the specific definition set $\mathbb{S}^1 \times [0,1]$; it only assumes it to be a two-dimensional subset of $\R^3$. Moreover, the result is also expected to remain valid in dimensions higher than two, although the proof would then require more delicate topological arguments.

Finally, in Section~\ref{section rien fixe}, Theorem~\ref{main th} is proved by combining these results, in particular the Hölder regularity and the preceding Lemma.

\textbf{Acknowledgments.}
 I would like to thank Antoine Lemenant for suggesting this problem to me as well as Matthieu Bonnivard for their help in this project. 
This work was partially supported by the IUF grant of Antoine Lemenant and by the ANR project STOIQUES.

\section{Existence and regularity of a minimizer \texorpdfstring{$u$}{u} for \texorpdfstring{$\ell \in \curve^\Lambda$}{l in Hom lambda} fixed}
\label{section l fixe}

\subsection{Existence and uniqueness with fixed $\ell$}

\begin{proposition}
\label{prop : holder regularity}There exits a constant $\varepsilon_0>0$, depending only on $\Ccal$ and $\Ccal_0$, such that for all $\varepsilon < \varepsilon_0$ the following property holds. \\
For any given $\ell \in \curve^\Lambda(\gamma_0, \gamma_1)$, there exists a unique minimizer $u$ of $E_\varepsilon^\ell := E_\varepsilon(\cdot,\ell)$ which is globally $C^{0,\alpha}$: for all $0< \alpha <1$ there exists a constant $C_\alpha$ such that 
\[\|u\|_{C^{0,\alpha}(\mathcal{C})} \leqslant C_\alpha \frac{1+\Lambda c_\varepsilon ^{-1}}{\varepsilon^{\alpha}}.\]
\end{proposition}

 The proof of this Proposition is divided in several lemmas. The first Lemma is a classical result on integral over upper-Ahlfors regular surfaces. 

\begin{lemma}
\label{lem 2.3}
Let $\ell \in \curve^\Lambda(\gamma_0, \gamma_1)$. We introduce the application
\[ \begin{array}{cccc}
B[\ell]  : &H^1(\mathcal{C}) \times H^1(\mathcal{C})  &\longrightarrow& \R, \\
  & (u,v) &\longmapsto &\int_{S_\ell}uv \dm \Haus.
\end{array}\]
$B[\ell]$ is a symmetric nonnegative bilinear and continuous form on $H^1(\mathcal{C})$ with the estimate 
\[\|B[\ell]\| \leqslant C \Lambda,\] where $C>0$ is a constant only depending on $\mathcal{C}$.

\end{lemma}

\begin{proof}
Clearly $B[\ell]$ is a nonnegative symmetric bilinear form. Hence, only the continuity remains to be proven. We denote by $\mu$ the finite measure $\Haus \mres S_\ell$.

\emph{Step 1.} Let $x\in \R^3$ and $r>0$ be such that $S_\ell \cap B(x,r) \neq \emptyset$. Fix $z\in S_\ell \cap B(x,r)$ then $S_\ell \cap B(x,r) \subset S_\ell \cap B(z,2r)$. Thus, the upper Ahlfors regularity of $S_\ell$ yields 
\begin{align*}
    \mu(B(x,r)) \leqslant \mu(B(z,2r)) \leqslant \Lambda \pi 4r^2,
\end{align*}
which leads to $\sup\left \{\frac{\mu(B(x,r))}{r^2}, r>0, x\in \R^3\right\} \leqslant 4\Lambda \pi$. Thus, from \cite[Theorem 5.12.4]{Zi} there exists a constant $C>0$, such that for all $u\in BV(\R^3)$ 

\[\left|\int u \dm \mu\right| \leqslant C\Lambda \|u\|_{BV(\R^3)}.\]

In particular, this yields for $w \in W^{1,1}(\R^3) \subset BV(\R^3)$ that
\begin{equation}
    \label{eq lem 2.3}
    \int |w|\dm \mu = \int_{S_\ell} |w|\dm \Haus \leqslant C\Lambda \left(\|w\|_{L^1} + \int_{\R^3} |\nabla w|\dm x \right).
\end{equation}

Hence, $w$ (or to be more exact the precise representative of $w$) is in $L^1(\mu)$ for all $w \in W^{1,1}(\R^3)$. 

\emph{Step 2.} Since $\C$ is assumed to be convex, there exists a continuous linear extension of $u$ on $\R^3$
\[ \begin{array}{ccc}
 H^1(\mathcal{C}) &\longrightarrow &H^1(\R^3),  \\
   u&\longmapsto &\overline{u}.
\end{array}\]
For any $u,v \in H^1(\C)$, $\overline{u} \, \overline{v} \in W^{1,1}(\R^3)$ and \textit{Step 1} yields $\overline{u} \, \overline{v} \in L^1(\mu)$, with the estimate
\begin{align*}
    |B[\ell](u,v)| &\leqslant \int_{S_\ell}|uv|\dm \Haus 
    = \int_{S_\ell}|\overline{u} \, \overline{v}|\dm \Haus 
    \leqslant C\Lambda \left(\|\overline{u} \, \overline{v}\|_{L^1} + \int_{\R^3}|\nabla(\overline{u} \, \overline{v})|\dm x\right) .
\end{align*}
Thus, since the extension on $\R^3$ is linear continuous, Cauchy-Schwarz inequality leads to the continuity of $B[\ell]$ ,
\[|B[\ell](u,v)| \leqslant \Lambda C_\C \|u\|_{H^1(\C)}\|v\|_{H^1(\C)},\]
which achieves the proof of this Lemma. 
\end{proof}

Hence, we can rewrite the energy functional as follows:

\[E^{\ell}_{\varepsilon}(u) := \varepsilon \int_{\mathcal{C}} |\nabla u|^{2} dx + \frac{1}{4\varepsilon}\int_{\mathcal{C}}(1-u)^{2} dx + \frac{1}{c_{\varepsilon}} B[\ell](u,u) +\frac{\delta_\varepsilon}{c_\varepsilon}\Haus(S_\ell).\]

\begin{proposition}
For any given $\ell \in \curve^\Lambda(\gamma_0, \gamma_1)$, there exists a unique $u \in H^1(\mathcal{C})$ minimizer of $E_\varepsilon^\ell$. 
\end{proposition}

\begin{proof}
Lemma~\ref{lem 2.3} ensures that the energy  $E_\varepsilon^\ell$ is lower semicontinuious with respect to the weak convergence in $H^1(\C)$ and that $E_\varepsilon^\ell(u)$ is finite for every $u\in H^1(\C)$. The direct method of calculus of variations yields therefore, the existence of a minimizer $u$. Moreover, owing to the strict convexity of the functional  $E_\varepsilon^\ell$ we obtain the uniqueness of $u$.
\end{proof} 

Note that the minimizer $u$ of $E_\varepsilon^\ell$ satisfies the Euler-Lagrange equation

\begin{equation}
\label{eq Euler Lagrange}
\begin{cases}
-\Delta u = \frac{1}{4\varepsilon^2}(1-u)- \frac{1}{\varepsilon c_\varepsilon}B[\ell](u,\cdot) & \text{ in } H^{-1}(\mathcal{C}),\\
u = 1 & \text{ on } \partial \mathcal{C} .
\end{cases}
\end{equation}

\subsection{Hölder estimates for $u$ when $S_\ell$ is fixed and upper Ahlfors regular}

Now that we have established the existence and uniqueness of a minimizer $u$ we focus on proving its Hölder regularity. In this Section, $u$ denotes the solution of the problem~\eqref{pb min l fixe} with $\ell \in~\curve^\Lambda(\gamma_0,\gamma_1)$ fixed.  

\begin{proposition}
\label{prop holder}
Let $\ell \in~\curve^\Lambda(\gamma_0,\gamma_1)$ and let $u$ be the minimizer of $E_\varepsilon^\ell$. Then, $u$ is $\alpha$-Hölder continuous for all $0<\alpha<1$ with the estimate:
\[\|u\|_{C^{0,\alpha}(\mathcal{C})} \leqslant C_\alpha \frac{1+\Lambda c_\varepsilon ^{-1}}{\varepsilon^{\alpha}}.\]
\end{proposition}

We start by establishing that $u$ is positive and bounded by $1$ almost everywhere. Then, since $u \in H^1(\C)\cap L^\infty(\C)$ and satisfies outside of the surface $S_\ell$ the elliptic equation 
\[-\Delta u = \frac{1}{4\varepsilon^2}(1-u) \text{ in } \mathcal{D}'(\C\setminus S_\ell),\]
the classical elliptic regularity theory will yield the smoothness of $u$ inside the domain and outside of the surface $S_\ell$. This is done in the following Lemma.

\begin{lemma}
\label{lem 2.4}
Let $\ell \in \curve^\Lambda(\gamma_0, \gamma_1)$, the associated solution $u$ satisfies $0 \leqslant u \leqslant 1$ almost everywhere in $\mathcal{C}$ and $u \in C^\infty(\mathcal{C}\setminus S_\ell)$.
\end{lemma}

\begin{proof}
We begin by showing that $0 \leqslant u \leqslant 1$ almost everywhere in $\mathcal{C}$. To that aim, we introduce the Lipschitz function $f:t \mapsto \max(0,\min(t,1))$ and we define $v :=f \circ u$, so that $0\leqslant v \leqslant 1$. It is well known that since $u \in 1+H^1_0(\C)$, we also have $v \in 1+H^1_0(\C)$ and moreover $|\nabla v| \leqslant |\nabla u|$ almost everywhere in $\C$. Similarly, $u^2 \in W^{1,1}(\C)$ yields $f\circ u^2 \in W^{1,1}(\C)$. Furthermore, because $v^2 \leqslant f\circ u^2 \leqslant u^2$ almost everywhere in $\C$, we get the inequality on the Lebesgue representative $(v^2)^* \leqslant (u^2)^*$ $\Haus$-almost everywhere and thus 
\[B[\ell](v,v) \leqslant B[\ell](u,u).\]
This leads to $E_\varepsilon^\ell(v) \leqslant E_\varepsilon^\ell(u) $ and the uniqueness of $u$ allows us to conclude that $u=v$ almost everywhere.

Now, notice that $u \in H^1(\C)\cap L^\infty(\C)$ and satisfies the following elliptic equation in $\mathcal{D}'(\C\setminus S_\ell)$
\[-\Delta u = \frac{1}{4\varepsilon^2}(1-u).\]

Thus, the standard elliptic regularity theory for bounded weak solutions (see for instance \cite[Corollary 8.11]{gilbargelliptic}) yields the conclusion that $u \in C^\infty(\C \setminus S_\ell)$.
\end{proof}

Hence, it remains to study the regularity at the boundary of the domain and on the surface $S_\ell$. To this end,  we begin by establishing an estimate for the gradient of $ u$ in both regions. We first address the estimate at the boundary of the domain, which corresponds to being far from the surface $S_\ell \in \overline{\Ccal_0}$. The method used here is standart in boundary elliptic regularity and relies on a barrier function.
In what follows, $C$ will denote a universal constant, that is, a constant depending only on the domains $\Ccal_0$ and $\Ccal$, and the dimension $N=3$.
\begin{lemma}
\label{lem estimation v proche du bord}
Let $\ell \in \curve^\Lambda(\gamma_0, \gamma_1)$, $u$ be the associated solution and $x_0 \in \mathcal{C} \setminus S_\ell$ be such that $\dist(x_0,S_\ell) \geqslant 10 \varepsilon$. Then,

\begin{equation}
    0 \leqslant 1-u(x_0) \leqslant \exp\left(-\frac{C \dist(x_0,S_\ell)}{\varepsilon}\right), \text{ where } C\in (0,1) \text{ is universal}.
\end{equation}
\end{lemma}

\begin{proof}
We fix $R:= \frac{9}{10}\dist(x_0, S_\ell) \geqslant 9\varepsilon$. Consider the function $v :=1-u$, satisfying $0\leqslant v \leqslant 1$ almost everywhere, and solution of 
\[\begin{cases}
-4\varepsilon^2\Delta v +v = 0 & \text{in } B(x_0,R) \cap \C, \\
v = 0 & \text{on } B(x_0,R) \cap \partial \C.
\end{cases}\]
We then introduce the function $w(x) = \exp \left(\frac{|x-x_0|^2-R^2}{8\varepsilon R}\right)$. We compute the laplacian of this barrier function: 
\[\Delta w(x) = \left(\frac{|x-x_0|^2}{16\varepsilon^2 R^2}+\frac{6}{8R\varepsilon}\right)w(x) \leqslant \left(\frac{1}{16\varepsilon^2 }+\frac{1}{12\varepsilon^2}\right)w(x),\]
which is therefore solution of \[\begin{cases}
-4\varepsilon^2\Delta w +w \geqslant 0 & \text{in } B(x_0,R) \cap \C, \\
w=1 \geqslant 0 & \text{on } \partial B(x_0,R)\cap \C, \\
w \geqslant 0 & \text{on } B(x_0,R) \cap \partial \C.
\end{cases}\]
Thus, since $v$ is smooth outside $S_\ell$ and since $B(x_0,R)\cap \Ccal \subset \Ccal \setminus S_\ell$, the maximum principle (see \emph{e.g.} \cite[Theorem 3.3]{gilbargelliptic}) yields $v\leqslant w$ in $B(x_0,R)\cap \C$. Hence, the desired inequality
\[1-u(x_0) = v(x_0) \leqslant w(x_0) = \exp\left(-\frac{9 \dist(x_0,S_\ell)}{80\varepsilon}\right) =  \exp\left(-\frac{C \dist(x_0,S_\ell)}{\varepsilon}\right),\]
and the proof is complete.
\end{proof}

\begin{lemma}
\label{lem 2.7} There exist constants $\varepsilon_0 >0$ and $C>0$ depending only on $\Ccal$ and $\Ccal_0$ such that if $\varepsilon < \varepsilon_0$, $\ell \in \curve^\Lambda(\gamma_0, \gamma_1)$, $u$ is the associated solution and $x_0 \in \overline{\mathcal{C}} \setminus S_\ell$ satisfies $d(x_0,S_\ell) \geqslant 11 \varepsilon$, then we have the following estimate  

\begin{equation}
   |\nabla u(x_0)| \leqslant \frac{C}{\varepsilon}\exp\left(-\frac{C \dist(x_0,S_\ell)}{\varepsilon}\right).
\end{equation}
\end{lemma}

\begin{proof}
\emph{Step 1. (Interior estimate)}
Assume that $B(x_0,\varepsilon) \subset \C$. We define, for all $x \in B_1 := B(0,1)$ the function $w_\varepsilon(x) :=1-u(x_0+\varepsilon x)$. This function satisfies 
\begin{equation}
    \label{eq lem 2.7 step 1}
    \Delta w_\varepsilon = -\varepsilon^2 \Delta u(x_0+\varepsilon x) = \frac{1}{4}w_\varepsilon.
\end{equation}
In the following, $C>0$ stands for a universal constant that may vary from one line to another.
The assumption $\dist(x_0,S_\ell)\geqslant 11 \varepsilon$ leads to $\dist(x_0+\varepsilon x,S_\ell) > 10 \varepsilon$, for all $x\in B_1$. Thus, the previous Lemma yields
\begin{align*}
    0\leqslant w_\varepsilon(x) &\leqslant \exp\left(-\frac{C \dist(x_0+\varepsilon x, S_\ell)}{\varepsilon}\right) \leqslant C \exp\left(-\frac{C \dist(x_0, S_\ell)}{\varepsilon}\right).
\end{align*}

Since $w_\varepsilon$ solves \eqref{eq lem 2.7 step 1} \cite[Theorem 3.9]{gilbargelliptic} implies that, for all $x\in B_1$,
\[|\nabla w_\varepsilon(x)|\dist(x,\partial B_1) \leqslant C(\|w_\varepsilon \|_{L^\infty(B_1)} + \sup_{y \in B_1} \dist(y,\partial B_1)^2|\frac{1}{4}w_\varepsilon(y)|)\leqslant C\|w_\varepsilon \|_{L^\infty(B_1)}. \]
In particular, for $x\in B_{1/2}$
\[|\nabla w_\varepsilon(x)| \leqslant C\|w_\varepsilon \|_{L^\infty(B_1)}. \]
Hence, for all $x\in B_{1/2}$
\[|\nabla w_\varepsilon(x)| \leqslant C\exp\left(-\frac{C \dist(x_0, S_\ell)}{\varepsilon}\right).\]
Noticing that $|\nabla w_\varepsilon(0)| = \varepsilon |\nabla u (x_0)|$ allows us to conclude that
\[ |\nabla u(x_0)| \leqslant \frac{C} {\varepsilon}\exp\left(-\frac{C \dist(x_0,S_\ell)}{\varepsilon}\right).\]

\emph{Step 2. (Boundary estimate)} We denote by $\eta_0 := \min\{\dist(x,\partial \mathcal{C}) \text{ for } x\in \mathcal{C}_0\} >0$, and consider a smooth intermediary open set $\mathcal{C}_1$ such that 
\[\overline{\Ccal_0}\subset \Ccal_1 \text{, } \overline{\Ccal_1}\subset \Ccal\text{ and } \min(\dist(\partial \Ccal_1, \partial \Ccal_0),\dist(\partial \Ccal_1,\partial \Ccal)) \geqslant \eta_0/4.\]
We denote by $U$ the open set $U := \Ccal \setminus \overline{\Ccal_1}$ and by $v$ the function $v :=1-u$, solution of 
\[\begin{cases}
\Delta v = \frac{1}{4\varepsilon^2}v & \text{in } U, \\
v = 0 & \text{on } \partial \C.
\end{cases}\]
Assume that $\varepsilon$ is small enough so that $10\varepsilon \leqslant \eta_0/4$. This means that all $x_0 \in \overline{U}$ are at distance at least $\eta_0/4  \geqslant 10\varepsilon$ from $S_\ell$. Therefore, for all $x_0 \in \overline{U}$ Lemma~\ref{lem estimation v proche du bord} yields 
\[0 \leqslant v(x_0) \leqslant C  \exp\left(-\frac{C \dist(x_0,S_\ell)}{\varepsilon}\right) \leqslant  C \exp\left( \frac{-C \eta_0}{4 \varepsilon}\right).\]
Indeed, for $x\in \partial \Ccal$ we have the boundary condition $v(x)=0$. The estimate of Lemma~\ref{lem estimation v proche du bord} can therefore be extended to $\overline{\Ccal}\setminus S_\ell$. Thus, we have a bound on the $L^\infty$ norm of $v$ in $\overline{U}$:
\begin{equation}
    \| v\|_{L^\infty(\overline{U})} \leqslant  C \exp\left( \frac{-C \eta_0}{4 \varepsilon}\right).
\end{equation}

Then, since $\dist(\partial \Ccal_1, \partial \Ccal)$ is assume to be larger than $\eta_0/4 \geqslant 10 \varepsilon > \varepsilon$, any point $x_0 \in \partial \Ccal_1$ is an interior point in the sense of \emph{Step 1}. Thus, \emph{Step 1} yields 
\[ |\nabla v(x_0)| = |\nabla u(x_0)| \leqslant \frac{C}{\varepsilon} \exp\left( \frac{-C \dist(x_0,S_\ell)}{ \varepsilon}\right) \leqslant  \frac{C}{\varepsilon} \exp\left( \frac{-C \eta_0}{4 \varepsilon}\right). \]
Hence, 
\begin{equation}
    \| \nabla v\|_{L^\infty( \partial \Ccal_1)} \leqslant   \frac{C}{\varepsilon} \exp\left( \frac{-C \eta_0}{4 \varepsilon}\right).
\end{equation}

Again, we can assume that $\varepsilon$ is small enough so that $\varepsilon <1$ and thus $1 < \frac{1}{\varepsilon}$. The previous estimates lead then to an estimate on the $C^1$ norm of $v$ on $\partial \Ccal_1$:
\begin{equation}
    \| v\|_{C^1(\partial \Ccal_1)} \leqslant   \frac{C}{\varepsilon} \exp\left( \frac{-C \eta_0}{4 \varepsilon}\right).
\end{equation}

Finally, because the boundary $\partial U = \partial \Ccal \cup \partial \Ccal_1$ is smooth (since both $\Ccal$ and $\Ccal_1$ are assumed to be smooth domains) and $v$ is a solution of a Poisson problem on $U$, the classical regularity theory for solution of Poisson problem (see for instance \cite[Theorem~8.33]{gilbargelliptic}) yields the existence of a constant $C>0$ depending only on $U$ \emph{i.e.} $C = C(\eta_0,\Ccal,\Ccal_0)$ such that 
\begin{align*}
     \|\nabla v\|_{L^\infty(\overline{U})} \leqslant \|v\|_{C^1(\overline{U})} &\leqslant C(\|v\|_{L^\infty(\overline{U})} + \| v\|_{C^1(\partial \Ccal_1)} + \frac{1}{\varepsilon^2}\|v\|_{L^\infty(U)}),\\
     &\leqslant C (1 + \frac{1}{\varepsilon} + \frac{1}{\varepsilon^2})\exp\left( \frac{-C \eta_0}{4 \varepsilon}\right),\\
       &\leqslant \frac{C}{\varepsilon^2}\exp\left( \frac{-C\eta_0}{4 \varepsilon}\right).
\end{align*}

Then, up to taking $\varepsilon$ even smaller (with respect to $\eta_0$), we can assume that $\frac{1}{\varepsilon}\exp\left( \frac{-C \eta_0}{8 \varepsilon}\right) \leqslant 1$. Therefore, we get
\begin{align*}
     \|\nabla v\|_{L^\infty(\overline{U})} \leqslant \frac{C}{\varepsilon}\exp\left( \frac{-C \eta_0}{8\varepsilon}\right).
\end{align*}
Notice that, for $x_0 \in \overline{U}$, $\dist(x_0,S_\ell) \leqslant \diam(\Ccal) = \frac{\diam(\Ccal)}{\eta_0}\eta_0$. Hence, for all $x_0 \in \overline{U}$, we get the desired estimate on the gradient: 
\begin{align*}
     |\nabla v(x_0)| &\leqslant \|\nabla v\|_{L^\infty(\overline{U})} 
      \leqslant  \frac{C}{\varepsilon} \exp\left( \frac{-C \eta_0}{8\varepsilon}\right) \leqslant \frac{ C }{\varepsilon} \exp\left( \frac{-C\eta_0}{8 \diam(\Ccal)}\frac{\dist(x_0,S_\ell)}{ \varepsilon}\right),\\
      &= \frac{ C}{\varepsilon} \exp\left( \frac{-C \dist(x_0,S_\ell)}{ \varepsilon}\right),
\end{align*}
which achieves the proof of this Lemma.
\end{proof}

We now turn to the gradient estimate on the surface $S_\ell$. 

\begin{lemma}
\label{lem 2.10}
Recall the definition of the parameter $\eta_0 = \dist(\partial \Ccal_0,\partial \Ccal)$ introduced in the previous Lemma.
Let $\ell \in \curve^\Lambda(\gamma_0, \gamma_1)$, $0< \rho < \min(1,\frac{\eta_0}{4})$ and $x_0 \in \mathcal{C}$ such that $\dist(x_0,\partial \mathcal{C}) \geqslant \frac{\eta_0}{2} $. Then, $u \in W^{1,p}(B(x_0,\rho))$ for all $3\leqslant p < \infty$ and we have the following estimate on the rescaled function $u_\rho (x) := u(x_0+\rho x)$
\begin{equation}
    \|\nabla u_\rho\|_{L^p(B_1)} \leqslant C_p\left( \frac{\rho^2}{\varepsilon^2}+1+\frac{\Lambda \rho}{c_\varepsilon \varepsilon}\right).
\end{equation}
\end{lemma}

\begin{proof}
\emph{Step 1.}
We introduce the distribution $T_\rho$ defined by
\[\langle T_\rho, \varphi \rangle = B[\ell](u,\varphi_\rho) = \int_{S_\ell}u \varphi_\rho \dm \Haus,\]
where $\varphi_\rho(x) := \varphi\left(\frac{x-x_0}{\rho}\right)$ and $\varphi \in C^\infty_c(\R^3)$. Let $\varphi \in C^\infty_c(B_ 2)$ then, by definition, $\varphi_\rho \in C^\infty_c(B(x_0,2\rho))$ and Lemma~\ref{lem 2.4} yield
\begin{align*}
    |\langle T_\rho, \varphi \rangle| &\leqslant \int_{S_\ell}|u_\ell||\varphi_\rho|\dm \Haus \leqslant \int_{S_\ell}|\varphi_\rho|\dm \Haus.
\end{align*}
Hence, with \eqref{eq lem 2.3} we obtain
\[|\langle T_\rho, \varphi \rangle| \leqslant C\Lambda \left(\int_{B(x_0,2\rho)}|\nabla \varphi_\rho|\dm x + \|\varphi_\rho\|_{L^1(B(x_0,2\rho))} \right) = C\Lambda \left( \rho^2\int_{B_2}|\nabla \varphi|\dm x + \rho^3\|\varphi\|_{L^1(B_2)} \right).\]
Thus, applying the Hölder inequality leads, for all $1\leqslant q \leqslant 3$, to $T_\rho \in W^{-1,p}(B_2) $. And since $\rho <1$ we have the estimate \[\|T_\rho\|_{W^{-1,p}(B_2)} \leqslant C\Lambda \rho^2.\]

\emph{Step 2.} The assumptions that $\dist(x_0,\partial \C) \geqslant \frac{\eta_0}{2}$ and $\rho < \frac{\eta_0}{4}$ yield $B(x_0,2\rho)\subset \C$. We now distinguish two cases whether $4\leqslant p < \infty$ or $3 \leqslant p <4$.

\emph{Case 1.} Assume that $4\leqslant p < \infty$. Since, $T_\rho \in W^{-1,p}(B_2)$ there exists a function $f\in L^p(B_2,\R^3)$ such that $\Div f = T_\rho$ in $\mathcal{D}'(B_2)$ (see \emph{e.g.} \cite[Sections 3.7 to 3.14]{adams2003sobolev}) such that 
\[C_p^{-1}\|T_\rho\|_{W^{-1,p}(B_2)} \leqslant \|f\|_{L^p(B_2,\R^3)} \leqslant C_p\|T_\rho\|_{W^{-1,p}(B_2)} .\] 
Thus \emph{Step 1} leads to $\|f\|_{L^p(B_2,\R^3)} \leqslant C_p \Lambda \rho^2$. Moreover, by classical elliptic theory (see for instance \cite[Theorem 19.15 and Lemma 19.17]{gilbargelliptic}) there exists a unique $\xi \in W^{2,p}(B_2,\R^3)\cap W^{1,p}_0(B_2,\R^3)$ solution of 
\[\begin{cases}
-\Delta \xi = f &\text{in } B_2,\\
\xi = 0 &\text{on } \partial B_2,
\end{cases}\]
such that $\|\xi \|_{W^{2,p}(B_2,\R^3)} \leqslant C_p \|f\|_{L^p(B_2,\R^3)}\leqslant C_p \Lambda \rho^2$.
And we introduce $v_\rho := \Div \xi \in W^{1,p}(B_2)$, which therefore satisfies in the distributional sense 
\[-\Delta v_\rho =\mathrm{div} f = T_\rho, \text{ and } \|v_\rho\|_{W^{1,p}(B_2)} \leqslant C_p \|T_\rho\|_{W^{-1,p}(B_2)} \leqslant C_p \Lambda\rho^2.\]
Thus, Sobolev embedding Theorem (see \emph{e.g.} \cite[Theorem 4.12]{adams2003sobolev}) yields $v_\rho \in L^\infty(B_2)$ with $\|v_\rho\|_{L^\infty(B_2)}\leqslant C_p \Lambda \rho^2$, since $3<p$.

\emph{Case 2.} Otherwise, if $3\leqslant p <4$, we define $v_\rho$ with $p=4$ so that the same argument can be applied. Namely, $v_\rho \in W^{1,4}(B_2)\cap L^\infty(B_2) \subset W^{1,p}(B_2)\cap L^\infty(B_2)$ with the same estimates $\|v_\rho\|_{W^{1,p}(B_2)} \leqslant \|v_\rho\|_{W^{1,4}(B_2)} \leqslant C_p \Lambda\rho^2$ and $\|v_\rho\|_{L^\infty(B_2)}\leqslant C \Lambda \rho^2$.

\medskip
\emph{Step 3.} We define the rescaled function $u_\rho$ on $B_2$ by $u_\rho(x) := u(x_0+\rho x)$. Then, for all $\varphi \in C^\infty_c(B_2)$ 
\begin{align*}
    \int_{B_2}\nabla u_\rho \cdot \nabla\varphi \dm x &= \frac{1}{\rho}\int_{B(x_0,2\rho)}\nabla u \cdot \nabla\varphi_\rho \dm x,\\
    &= \frac{1}{\rho4\varepsilon^2}\int_{B(x_0,2\rho)}(1- u)\varphi_\rho \dm x - \frac{1}{\rho c_\varepsilon \varepsilon}B[\ell](u,\varphi_\rho),\\
    &= \frac{\rho^2}{4\varepsilon^2}\int_{B_2}(1- u_\rho)\varphi \dm x - \frac{1}{\rho c_\varepsilon \varepsilon}\langle T_\rho,\varphi \rangle.
\end{align*}
Hence, $u_\rho$ satisfies the following PDE in $\mathcal{D}'(B_2)$
\[-\Delta u_\rho =  \frac{\rho^2}{4\varepsilon^2}(1- u_\rho) - \frac{1}{\rho c_\varepsilon \varepsilon} T_\rho.\]
We denote $w_\rho := u_\rho + \frac{1}{\rho c_\varepsilon \varepsilon}v_\rho \in H^1(B_2)\cap L^\infty(B_2)$. And we deduce that 
\[-\Delta w_\rho = \frac{\rho^2}{4\varepsilon^2}(1- u_\rho) \text{ in } \mathcal{D}'(B_2). \]
Thus, \cite[Corollary 8.36]{gilbargelliptic} yields that $w_\rho \in C_{\mathrm{loc}}^{1,\alpha}(B_2)$ for all $\alpha >0$ and 
\begin{align*}
    \|\nabla w_\rho\|_{L^\infty(B_1)} \leqslant \|w_\rho\|_{C^{1,\alpha}(B_1)} &\leqslant C\left(\| w_\rho\|_{L^\infty(B_1)} +\frac{\rho^2}{4\varepsilon^2}\|1- u_\rho\|_{L^\infty(B_1)}\right),\\
    &= C\left(\| u_\rho\|_{L^\infty(B_1)}+ \frac{1}{\rho c_\varepsilon \varepsilon}\| v_\rho\|_{L^\infty(B_1)} +\frac{\rho^2}{4\varepsilon^2}\right),\\
    &\leqslant C\left(1 + \frac{\Lambda \rho}{c_\varepsilon \varepsilon} + \frac{\rho^2}{4\varepsilon^2}\right).
\end{align*}
Finally, we can conclude that $u_\rho = w_\rho - \frac{1}{\rho c_\varepsilon \varepsilon}v_\rho \in W^{1,p}(B_1)$ with the desired estimate
\[ \|\nabla u_\rho\|_{L^p(B_1)} \leqslant C_p\left( \frac{\rho^2}{\varepsilon^2}+1+\frac{\Lambda \rho}{c_\varepsilon \varepsilon}\right).\]
This concludes the proof of this Lemma.
\end{proof}

Consequently, we can establish the Hölder regularity of the solution $u$ everywhere in $\overline{\Ccal}$.

\begin{proof}[Proof of Proposition~\ref{prop holder}]
We can assume without loss of generality that $11\varepsilon < \frac{\eta_0}{4}$, where we recall that $\eta_0 = \dist(\partial \C, \partial \C_0)$. Let $x,y \in \mathcal{C}$, we define $x_0 = \frac{x+y}{2} \in \Ccal$, since $\Ccal$ is assumed to be convex.
\begin{itemize}
    \item If $|x-y|\geqslant \varepsilon$, then we directly get 
    \[\frac{|u(x)-u(y)|}{|x-y|^\alpha} \leqslant \frac{2}{\varepsilon^\alpha},\]
    since from Lemma~\ref{lem 2.4} $\|u\|_{L^\infty} \leqslant 1$.
    \item Otherwise, if $|x-y|< \varepsilon$, we distinguish whether $x_0$ is closer to $\partial \mathcal{C}$ or $\partial \mathcal{C}_0$. First, assume that $\dist(x_0, \partial \mathcal{C}) \leqslant \frac{\eta_0}{2}$. From the Definition of $\eta_0 = \dist(\partial \Ccal, \partial \Ccal_0)$ and the assumption that $\varepsilon < 11 \varepsilon < \frac{\eta_0}{4}$ it yields that for all $z\in B(x_0,\varepsilon)$, $\dist(z,S_\ell) > \frac{\eta_0}{4} > 11 \varepsilon$. Thus, for any $z \in B(x_0,\varepsilon)\cap \overline{\C}$ Lemma~\ref{lem 2.7} yields 
        \[|\nabla u (z)| \leqslant\frac{C}{\varepsilon} \exp \left(-\frac{C \dist(z,S_\ell)}{\varepsilon}\right) < \frac{C}{\varepsilon} \exp \left(-11C\right) = \frac{C}{\varepsilon}.\]
    And thanks to the mean value Theorem, we can conclude that 
    \[\frac{|u(x)-u(y)|}{|x-y|^\alpha} \leqslant \frac{C}{\varepsilon^\alpha}.\]
    
    Now let's consider the case where $\dist(x_0, \partial \mathcal{C}) > \frac{\eta_0}{2}> \varepsilon$, in particular this implies that $B(x_0,\varepsilon) \subset \C$. In this case, Lemma~\ref{lem 2.10}, applied with $\rho = \varepsilon$ and $p = \frac{3}{1-\alpha}$ yields the estimate on the rescaled function $u_\varepsilon$
    \[\|\nabla u_\varepsilon\|_{L^p(B_1)} \leqslant C_p\left(1+ \frac{\Lambda}{c_\varepsilon}\right).\]
    Since we have shown that the $L^\infty$ norm of $u$ is smaller than $1$, we deduce that the $L^p$ norm of the scaled function $u_\varepsilon$ is bounded by the volume of the domain. And we conclude by applying the Sobolev embedding Theorem (see for instance \cite[Theorem 4.12]{adams2003sobolev}),
    \begin{equation}
        \|u_\varepsilon\|_{C^{0,\alpha}(B_1)} \leqslant C_\alpha\left(1+ \frac{\Lambda}{c_\varepsilon}\right).
    \end{equation}
    Hence, scaling back leads to 
    \begin{equation}
        \frac{|u(x)-u(y)|}{|x-y|^\alpha} \leqslant \frac{1}{\varepsilon^\alpha}  \|u_\varepsilon\|_{C^{0,\alpha}(B_1)}  \leqslant \frac{C_\alpha}{\varepsilon^\alpha}\left(1+ \frac{\Lambda}{c_\varepsilon}\right).
    \end{equation}
\end{itemize}
Finally, since $\|u\|_{L^\infty(\Ccal)} \leqslant 1$, we get the desired estimate on the Hölder norm of $u$ 
\[ \|u\|_{C^{0,\alpha}(\mathcal{C})} \leqslant C_\alpha \frac{1+\Lambda c_\varepsilon ^{-1}}{\epsilon^{\alpha}}.\]
This achieves the proof of the Hölder regularity of the solution $u$ of $\min E_\eps(u,\ell)$, when $\ell$ is fixed in $\curve^\Lambda$.
\end{proof}

\section{Existence of a minimizer \texorpdfstring{$\ell$}{l} in \texorpdfstring{$\curve^\Lambda$}{Hom lambda} for \texorpdfstring{$u$}{u} fixed}
\label{section u fixe}

Let $u\in H^1(\Ccal)\cap C(\overline{\Ccal})$ be fixed. This section is devoted to the study of the following problem:
	\begin{equation}
\inf_{\ell \in \curve^\Lambda(\gamma_0,\gamma_1)} \int_{\surf_{\ell}} (u^2 + \delta_\eps) d\Haus.  \label{probMin}
	\end{equation}

Notice that, since $u$ is fixed, minimizing the energy $E_\eps(u,\cdot)$ is equivalent to solving the problem~\eqref{probMin}.

As usual, the existence result requires both the lower semicontinuity of the functional and some compactness property. Here, we consider uniformly Lipschitz functions, so compactness is ensured by the Ascoli–Arzelà theorem. Therefore, the main task is to prove lower semicontinuity. This follows from Lemma~\ref{golab}, which establishes the lower semicontinuity of the $\Haus$-measure for surfaces defined as the images of Lipschitz maps. This lemma can be viewed as a generalization of Go\l{}\k{a}b’s theorem in dimension $2$, under the assumption that the surfaces are parametrized by Lipschitz maps.

\begin{proof}[Proof of Lemma~\ref{golab}]
We consider the sequence of measures $\mu_n:= \mathcal{H}^2\mres {S_{\ell_n}}$. Since the sequence $(\ell_n)$ is uniformly Lipschitz, the Lipschitz constants of $\ell_n$ are uniformly bounded by a constant, denoted by $\Lambda$. Therefore, \[\Haus(S_{\ell_n})\leqslant \Lip(\ell_n)^2 \Haus([0,1]\times \mathbb{S}^1) \leqslant \Lambda^2 \pi,\] which yields the uniform boundedness of the measures $\mu_n$. Thus, the sequence $(\mu_n)$  weakly converges to a measure $\mu$. Thanks to the lower-semi-continuity behavior with respect to the weak convergence of measures we know that, for all open set $A$
$$\mu(A) \leqslant \liminf_{n \to + \infty} \mu_n(A). $$

Therefore, it remains to prove that
\begin{eqnarray}
\mathcal{H}^2\mres {S_{\ell}} \leqslant \mu . \label{claim}
\end{eqnarray}

We will prove later that the 2-dimensional density of $\mu$ at almost every point in $S_\ell$ is larger than 1. Consequently, a classical fact of Geometric Measure Theory (see for instance \cite[Proposition 2.21]{AFP00}) yields that for any Borel set $A\subset S_\ell$, $\mathcal{H}^2\mres {S_{\ell}}(A) = \mathcal{H}^2(A) \leqslant \mu(A)$. Thus, for any Borel set $B$, $B\cap S_\ell$ is a Borel set contained in $S_\ell$ and we get 
\[\mathcal{H}^2\mres {S_{\ell}}(B) = \mathcal{H}^2(B\cap S_\ell) \leqslant \mu(B\cap S_\ell) \leqslant \mu(B),\] which proves the inequality~\eqref{claim}. Therefore, to conclude, it suffices to establish the following claim: 

\begin{equation}
\label{eq estimate on the density}
\limsup_{r\to 0} \frac{\mu(\overline{B}(x_0,r))}{\pi r^2}\geqslant 1, \text{ for almost every } x_0\in S_\ell.    
\end{equation}

The proof of \eqref{eq estimate on the density} will be done via two main steps.

\emph{Step 1.(Approximation via a tangent plane)}
Let $x_0\in S_\ell$ such that, $x_0=\ell(t_0)$ with $t_0$ a differentiability point of $\ell$ such that the differential $\dm_{t_0} \ell$ is non degenerate. Such a point exists almost everywhere, since $\ell$ is Lipschitz and therefore differentiable almost everywhere by Rademacher’s Theorem. Moreover, \cite[Lemma 2.96]{AFP00} implies that, almost everywhere, the image of the differential is two-dimensional. Consider the affine map $L : h \mapsto \ell(t_0) + \dm_{t_0}\ell(h)$.

Let $\eta >0$. Since, $\ell$ is differentiable in $t_0$, there exists a radius $r_0 >0$ such that, for all $h \in D(0,r_0)$, 
\begin{equation}
    \frac{|L(h) - \ell(t_0+h)|}{|h|}\leqslant \eta.
\end{equation}

 We define the portion of the tangent plane to $S_\ell$ at point $x_0=\ell(t_0)$ by $T_{r_0} = L(D(0,r_0))$. For simplicity, we assume that $T_{r_0}$ is horizontal, i.e., $T_{r_0}\subset \{(y_1,y_2,y_3) \; |\; y_3=0\}$.  Since $T_{r_0}$ is an ellipse, as the image of a disk under an non degenerate affine map, there exists a smaller radius $0<r'_0\leqslant r_0 $
such that $B(x_0,r)\cap  T_{r_0}=D(x_0,r)$ for all $r\leq r_0'$ (where $D(x_0,r)$ denotes the horizontal disk of radius $r$).

Let now $r\leq r_0'$ be fixed. Since the sequence $(\ell_n)$ uniformly converges to $\ell$, there exists $N>0$ such that for all 
$n>N$, 
\begin{equation}
    \|\ell_n -\ell \|_{\infty} \leqslant  \eta r. \label{proximite}
\end{equation}
Let $n>N$. We denote by 
$$H_r = t_0 + L^{-1}(\overline{D}(x_0,r)).$$
Since $d\ell_{t_0}$ is assumed to be non-degenerate, $H_r$ is an ellipse, being the pre-image of a disk under an affine map. Moreover, $H_r \subset \overline{D}(t_0,Cr)$, where $C$ is a constant that depends only on $d_{t_0}\ell$.  We deduce that, if $r < \frac{r_0}{C}$, then for any $y\in H_r$ and for all $n\geq N$,
\begin{equation}
    |L(y-t_0)-\ell_n(y)| \leqslant C\eta r + \eta r \leqslant \eta C'r, \label{estimation1}
\end{equation}
with $C'=C+1$, which we shall still denote by $C>0$ in the sequel. Since, $C$ does not depend on $r$, we fix $0<r<\min(r_0,\frac{r_0}{C})$. Hence, $\ell_n(H_r)$ is a surface at distance at most $\eta C r$ of $L(H_r-t_0) = \overline{D}(x_0,r)$, which yields that $\ell_n(H_r)\subset \overline{B}(x_0,r(1+C\eta))$. Thus,
\begin{eqnarray}
\mathcal{H}^2(S_{\ell_n}\cap \overline{B}(x_0,r(1+C\eta)))\geqslant \mathcal{H}^2(\ell_n(H_r)). \label{ett}
\end{eqnarray}

Now, let $P:\R^3\to \R^3$ denote the projection onto the horizontal plane $\{(y_1,y_2,y_3) \;|\; y_3=0\}$. We claim that $P(\ell_n(H_r))$ contains a disk of radius at least $(1-C\eta)r$. 
Let $\varphi$ be the smooth parametrization of the ellipse $\partial H_r$ by $\mathbb{S}^1$ with constant speed, namely $\varphi'(\omega) = \frac{\mathcal{H}^1(\partial H_r)}{2\pi}$, where $\mathcal{H}^1(\partial H_r)$ is the length of the ellipse $\partial H_r$. In other words, the curve $\varphi(\mathbb{S}^1)$ makes only one turn around the point $t_0$. In particular, this parametrization is injective and its speed never vanishes. Let also $(\omega_i)_{1\leqslant i\leqslant m}$ be a subdivision of $\mathbb{S}^1$ with $\omega_1 = \omega_m$ (and with $m \geqslant 17$). We denote by $(t_i)_{1\leqslant i\leqslant m}$ the points $\varphi(\omega_i)$ on $\partial H_r$. Recall that $L(\partial H_r) = C(x_0,r)$ is a circle. Since $L$ is affine, the parametrization $L \circ \varphi $ of $C(x_0,r)$ is smooth. Furthermore, since the differential $ \dm _{t_0}\ell$ is non-degenerate, $L$ is injective and its differential $\dm L = \dm _{t_0}\ell$ is also non-degenerate. Hence, the parametrization $L \circ \varphi $ is injective and has non-vanishing speed, ensuring that it traces the circle exactly once around $x_0$. Therefore, one can choose the subdivision $(\omega_i)$ such that $(L(t_i))$ is a regular subdivision of the circle $C(x_0,r)$, more precisely such that $\mathcal{H}^1(L([t_i, t_{i+1}])) = \frac{2\pi r}{m-1}$, and $|L(t_i) -L(t_{i+1})| \geqslant C r \eta
$. Notice that, we also assume, without loss of generality since $\eta$ can be chosen as small as we want, that $C \eta \leqslant \frac{\pi}{8}$. Thus we have constructed a parametrization of $\ell_n(\partial H_r)$ by $\mathbb{S}^1$ as $\Gamma = \ell_n \circ \varphi (\mathbb{S}^1)$, and naturally, $\Gamma$ is a continuous closed curve remaining $C\eta r$ close to the circle $L(\partial H_r)$. More preciselly, the parametrization $\ell_n \circ \varphi$ is $C\eta r$ close to the  parametrization of circle $L\circ \varphi$.

Then, consider the projection of this curve on the horizontal plane $\Tilde{\Gamma} = P(\Gamma)$. $\Tilde{\Gamma}$ is a closed curve contained in the (horizontal) disk $D(x_0,r(1+C\eta))$ and lies at most a distance $C\eta r$ from the circle $C(x_0,r)$. More precisely, $P\circ \ell_n \circ \varphi$ is $C\eta r $-close to $ L \circ \varphi$. We now justify that every half-line emanating from $x_0$ intersects the closed curve $\Tilde{\Gamma}$. Assume by contradiction that, there exists a half line that does not intersect $\Tilde{\Gamma}$. Nevertheless, this half-line intersects the circle $C(x_0,r)$ at a point denoted by $y_0 = L(s_0)$, with $s_0 \in \partial H_r$. Since $(L(t_i))$ is a subdivision of the circle, there exists $1\leqslant i < m$ such that $s_0 \in [t_i,t_{i+1}[ $ then $L(t_{i-3})$ and $L(t_{i+4})$ are in two different connected components of $C(x_0,r) \setminus (x_0,y_0)$, and more precisely at distance at least $\frac{3}{2}C\eta r$ from the line $(x_0,y_0)$, with the convention that $t_{m+k} = t_{k+1}$ (see Lemma~\ref{lem subdivision} for the proof of this fact, with $\psi = L\circ \varphi$ and $\delta = C\eta$). This implies that $P(\ell_n(t_{i-3}))$ and $P(\ell_n(t_{i+4}))$ lies in two distinct components of $D(x_0,r(C\eta+1)) \setminus (x_0,y_0)$. Moreover, by construction, $P(\ell_n([t_{i-3},t_{i+4}]))$ connects $P(\ell_n(t_{i-3}))$ to $P(\ell_n(t_{i+4}))$ and, by assumtion, does not intersect the half line $[x_0,y_0)$. Therefore, there exists $s_1 \in [t_{i-3},t_{i+4}]$ such that $P(\ell_n(s_1))$ intersects the half line $(x_0,y_0) \setminus [x_0,y_0)$, and for all $t\in[t_{i-3},t_{i+4}]]$, $L(t)$ remains at a distance at least $\frac{3}{2}C\eta r$ from this half line $(x_0,y_0)\setminus [x_0,y_0)$ (see Figure~\ref{fig1}). In particular, it yields $$|L(s_1)-P(\ell_n(s_1))|\geqslant \frac{3}{2}C\eta r,$$ which contradicts \eqref{estimation1}.

\begin{figure}[ht]
    \centering
    \includegraphics[width = 0.5\textwidth]{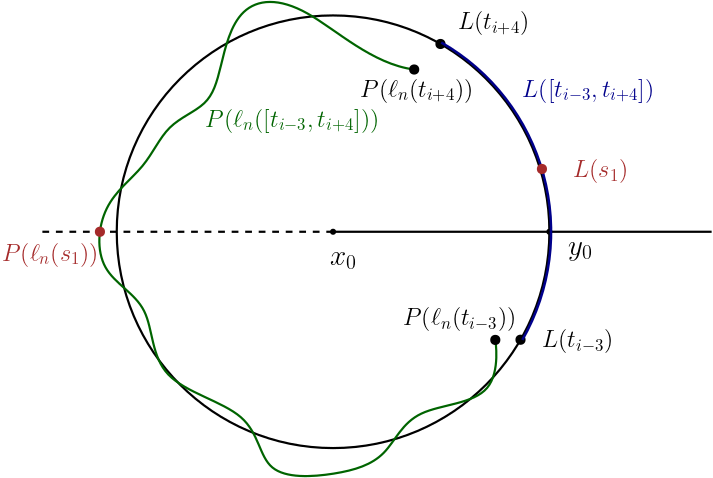}
    \caption{\centering Argument by contradiction to prove that every half-line emanating from $x_0$ intersects the closed curve $\Tilde{\Gamma}$.}
    \label{fig1}
\end{figure}

This yields that the projection of $\Tilde{\Gamma}$ onto the circle $C(x_0,r(1-C\eta))$, is surjective. Let $\pi$ denote this projection on the circle $C(x_0,r(1-C\eta))$, consider

\[\Phi(s,y) =  s \Tilde{\gamma}(y)  +(1-s) \pi( \Tilde{\gamma}(y)), \text{ where } \Tilde{\gamma} = P\circ \ell_n \circ \varphi.\]

$\Phi$ is an homotopy between $\pi \circ \Tilde{\gamma}$ and $\Tilde{\gamma}$. Thus, the curve $\Tilde{\Gamma}$ is homotopy equivalent to the circle $C(x_0,r(1-C\eta)) = \pi(\Tilde{\Gamma})$.

Let us conclude the proof of the claim that, the (open) disk $D(x_0,r(1-C\eta))$ is contained in $P(\ell_n(H_r))$. Assume by contradiction, that there is a point $y$ in the disc $D(x_0,r(1-C\varepsilon))$ that does not belong to the surface $P(\ell_n(H_r))$. On one hand, we have a closed curve $\Tilde{\Gamma}$ in $P(\ell_n(H_r))$ homotopy equivalent to a circle non contractible in $P_{r_0}\setminus \{y\}$ since $y$ is in the open disc. In other words, $\Tilde{\Gamma}$ has a non trivial $\pi_1$. On the other hand however, $\partial H_r$ is an ellipse that is contractible to its center. Therefore, since $P \circ \ell_n$ is continuous $\Tilde{\Gamma} = P(\ell_n(\partial H_r))$ is also contractible in $P(\ell_n(H_r)) \subset T_{r_0} \setminus \{y\}$ \emph{i.e.}, $\Tilde{\Gamma}$ has a trivial $\pi_1$. Hence, the contradiction and thus the claim is proved.

\emph{Step 2.(Density estimate)}
It follows from \emph{Step 1} that 
\[\mathcal{H}^2(\ell_n(H_r))\geqslant \mathcal{H}^2(P(\ell_n(H_r)))\geqslant \pi r^2(1-C\eta)^2,\]
and using \eqref{ett} this yields,
$$\mu_n(\overline{B}(x_0,r(1+C\eta))\geqslant \pi r^2(1-C\eta)^2.$$
Letting $n\to +\infty$ and taking the limsup we get 
$$\mu(\overline{B}(x_0,r(1+C\eta))\geqslant \limsup_n \mu_n(\overline{B}(x_0,r(1+C\eta))\geqslant \pi r^2(1-C\eta)^2.$$

Then, we let $r\to 0$, yielding
$$\limsup_{r\to 0} \frac{\mu(\overline{B}(x_0,r(1+C\eta)))}{\pi r^2 (1+C\eta)^2}\geqslant \frac{(1-C\eta)^2}{(1+C\eta)^2} .$$
We conclude by letting $\eta \to 0$ 
$$\limsup_{r\to 0} \frac{\mu(\overline{B}(x_0,r))}{\pi r^2}\geqslant 1,$$
and so follows the lemma.
\end{proof}

We now provide a technical proof of the geometric claim used in the previous Lemma.

\begin{lemma}
\label{lem subdivision}
Let $0<\delta <\frac{\pi}{8}$. Consider a circle $C(x_0,r)\in \R^2$ parametrized on $\mathbb{S}^1$ by a smooth injective function $\psi$ with non vanishing speed, which insures that the curve $\psi(\mathbb{S}^1)$ makes one and only one turn around the point $x_0$. Thus, there exists $(\omega_i)_{1\leqslant i \leqslant m}$ a subdivision of $\mathbb{S}^1$ such that $m \geqslant 17$, which is regular with respect to $\psi$, \emph{i.e.} for all $i$, $\mathcal{H}^1(\psi([\omega_i,\omega_{i+1}])) = \frac{\mathcal{H}^1(\psi(\mathbb S^1))}{m-1}=\frac{2\pi r}{m-1}$, and such that for all $i$, $|\psi(\omega_i) - \psi(\omega_{i-1})| \geqslant r\delta$. Let $y_0 = \psi(w)$ be an arbitrary point on the circle $C(x_0,r)$. Then, there exists an index $i$ such that $w \in [\omega_i,\omega_{i+1}[ $ and such that $\psi(\omega_{i-3})$ and $\psi(\omega_{i+4})$ lie in two different connected components of $C(x_0,r) \setminus (x_0,y_0)$   and at distance at least $\frac{3}{2} \delta r$ from the line $(x_0,y_0)$ (with the convention $\omega_{m+k} = \omega_{k+1}$).

\end{lemma}

\begin{proof}
By definition of the subdivision, there exists a unique index $i$ such that $w \in [\omega_i, \omega_{i+1}[$. First, we justify why $\psi(\omega_{i-3})$ and $\psi(\omega_{i+4})$ lie in two different connected components of $C(x_0,r) \setminus (x_0,y_0)$. Since $\psi([w,\omega_{i+4}]) \subset \psi([\omega_i,\omega_{i+4}])$, the assumption $m \geqslant 17$ yields that the  circular arc $\psi([w,\omega_{i+4}])$ has a smaller length than $4 \frac{2\pi r}{m-1} \leqslant r \frac{\pi}{2}$. Thus, $\psi( \omega_{i+4})$ remains in the same connected component as $\psi(\omega_{i+1})$. Similarly, $\psi(\omega_{i-3})$ remains in the same connected component as $\psi(\omega_{i-1})$ and we can conclude that $\psi(\omega_{i-3})$ and $\psi(\omega_{i+4})$ are in two different connected components of $C(x_0,r) \setminus (x_0,y_0)$.

Secondly, we need to show that $\psi(\omega_{i-3})$ and $\psi(\omega_{i+4})$ are  at distance at least $\frac{3}{2}\delta r$ from the line $ (x_0,y_0)$. We develop the proof for $\psi(\omega_{i+4})$, naturally by symmetry, the proof for $\psi(\omega_{i-3})$ is analogous. One can compute explicitly the distance, denoted by $d$, between $\psi(\omega_{i+4})$ and the line $(x_0,y_0)$.

\[d = r \sin \left(\frac{\mathcal{H}^1(\psi([w,\omega_{i+4})))}{r} \right).\]

The assumption on the subdivision yields 
\begin{align*}
 \mathcal{H}^1(\psi([w,\omega_{i+4}]))&\geqslant \mathcal{H}^1(\psi([\omega_{i+1},\omega_{i+4}])) = 3\mathcal{H}^1(\psi([\omega_{i+1},\omega_{i+2}]))\\
 &\geqslant 3|\psi(\omega_{i+1})-\psi(\omega_{i+2})| \geqslant 3 r \delta. 
\end{align*}
Furthermore, by concavity of $t\mapsto \sin(t)$ on $[0,\frac{\pi}{2}]$ we get
\[d \geqslant r \sin(3\delta) \geqslant r \frac{2}{\pi} 3\delta > \frac{3}{2} \delta r,\]
which achieves the proof of this Lemma.
\end{proof}

With the lower semi-continuity now established, the existence theorem follows directly.

\begin{proposition}
For all $u \in H^1(\C)\cap C(\overline{\Ccal})$ there exists a minimizer for the problem \eqref{probMin}. 
\end{proposition}
\begin{proof} Let $(\ell_n)$ be a minimizing sequence. Then, $(\ell_n) \subset \curve^\Lambda$ is a sequence of uniformly Lipschitz functions. Therefore, there exists a subsequence, still denoted by $(\ell_{n})$, which converges uniformly to $\ell \in \curve^\Lambda$. Thus, Lemma~\ref{golab} yields that, for all open set $A$
\[\liminf \Haus \mres S_{\ell_n}(A) \geqslant \Haus \mres S_\ell(A).\]
More prescisely, we have established that the sequence of measures $\mu_n:= \mathcal{H}^2\mres {S_{\ell_n}}$ weakly converges to a measure $\mu$, such that 
\begin{eqnarray}
\mathcal{H}^2\mres {S_{\ell}} \leqslant \mu . 
\end{eqnarray}

Besides, since $u^2+\delta_\varepsilon$ is a continuous function with constant value (equal to $1$) on the boundary of the domain $\mathcal{C}$, we can extend this function on $\R^3$. This extension is a continuous function with support in the compact $\overline{\Ccal}$. Moreover, since the measures $\mu_n$ and $\mu$ are also supported on the compact $\overline{\mathcal{C}}$, the duality between compactly supported continuous function and Radon measures yields 
\[\int_{S_\ell} (u^2+\delta_\varepsilon) \dm \Haus \leqslant \int_\Ccal (u^2+\delta_\varepsilon) \dm \mu =\lim \int_{\Ccal} (u^2+\delta_\varepsilon) \dm\mu_n = \liminf \int_{S_{\ell_n}} (u^2+\delta_\varepsilon) \dm \Haus,\]
because $u^2+\delta_\varepsilon$ is a positive function.

This conclude the proof of the lower semi-continuity of the functional $\ell \mapsto \int_{\surf_{\ell}} (u^2 + \delta_\eps) d\Haus$. 

Finally, to obtain the existence of minimizers for this functional, it remains to establish that
 $S_\ell$ is $\Lambda$-upper Ahlfors regular. This follows directly from Lemma~\ref{golab}: for all $x\in S_\ell$ and $r>0$, 
\[\Haus(B(x,r)\cap S_\ell) \leqslant \liminf \Haus(B(x,r)\cap S_{\ell_n}).\]
Since, for every $n$, the surface $S_{\ell_n}$ is $\Lambda$-Ahlfors regular, we get
\[\Haus(B(x,r)\cap S_\ell) \leqslant \Lambda \pi r^2.\]
Thus, $\ell \in \curve^\Lambda$ and the Proposition is proven.
\end{proof}

\section{Existence of a minimizer in the variable \texorpdfstring{$(u,\ell)$}{(u,l)} for \texorpdfstring{$\ell \in \curve^\Lambda$}{l in Hom lambda} }
\label{section rien fixe}

Combining the above results, we can now prove the global existence of a minimizer in both variables (Theorem~\ref{main th}), announced in the introduction.

%
%
%
%

\begin{proof}[Proof of Theorem~\ref{main th}]

Let $(u_n,\ell_n)$ be a minimizing sequence. We may assume that, for all $n$, $u_n$ is the solution of the problem of minimization in the $u$ variable with fixed $\ell_n$, for which the existence and uniqueness is guaranteed by Proposition~\ref{prop : holder regularity}. Indeed, when replacing $u_n$ by the minimizer in the $u$ variable, it remains a minimizing sequence. And Proposition~\ref{prop : holder regularity} yields that the Hölder norms of $u_n$ are uniformly bounded. Hence, Ascoli Theorem yields that $u_n$ converges to $u \in H^1(\Ccal)$, uniformly. Furthermore, $u$ is $\alpha$-Hölder continuous for all $0<\alpha<1$, with the desired estimate on the Hölder norm. Moreover, $\ell_n$ converges to $\ell \in \curve^\Lambda$ also uniformly. Since the Ambrosio-Tortorelli part of the functional $ \varepsilon\int_{\mathcal{C}} |\nabla u|^{2} dx + \frac{1}{4\varepsilon}\int_{\mathcal{C}}(1-u)^{2} dx $ is clearly lower semi-continuous, it remains to check the lower semi-continuity of the penalization term, i.e. that 

$$\int_{\surf_{\ell}} (u^{2} + \delta_\eps) d\Haus \leq \liminf\int_{\surf_{\ell_n}} (u_n^{2} + \delta_\eps) d\Haus.$$

To that aim, we will adapt the proof developed in the previous Section. 

We denote by $v_n := u_n^2 + \delta_\varepsilon$ and $v := u^2+ \delta_\varepsilon$. 
Consider the sequence of measures $\mu_n:= v_n \mathcal{H}^2\mres {S_{\ell_n}}$ which are uniformly bounded, since $v_n$ is uniformly bounded by $1+\delta_\eps$ and the measures $\mathcal{H}^2\mres {S_{\ell_n}}$ are uniformly bounded, for the same reasons as in the proof of Lemma~\ref{golab}. Thus, the sequence weakly converges to a measure $\mu$ and the lower-semi-continuity behavior with respect to weak convergence yields that 
$$\mu(U) \leq \liminf \mu_n(U), \text{ for all } U\subset \R^3 \text{ open set}. $$
Hence, it is enough to prove that 
\begin{eqnarray}
v \mathcal{H}^2\mres {S_{\ell}} \leqslant \mu . 
\end{eqnarray}
Since, $\mu$ is a Radon measure, Besicovitch derivation theorem insures that for $\Haus\mres S_\ell$-almost every $x$ the limit 
\[f(x) := \lim_{r\to 0} \frac{\mu(B(x,r))}{\omega_2 r^2}\] exists and moreover, $\mu = f\Haus\mres S_\ell + \mu \mres E$, where $E$ is a $\Haus\mres S_\ell$-negligible set. Since $\mu$ is a positive Radon measure this yields that $\mu \geqslant f \Haus\mres S_\ell$. Hence, it is enough to show that for almost all $x \in S_\ell$, $f(x) \geqslant v(x)$. 

By \emph{Step 1} of Lemma~\ref{golab}, for almost every $x\in S_\ell$ and all $\eta >0$ there exits a radius $r_0$ such that, for all $r<r_0$, there exists a rank $N$ satisfying that for all $n>N$,
\[\Haus(S_{\ell_n}\cap \overline{B}(x,r(1+C\eta))) \geqslant \pi r^2(1-C\eta)^2.\]

Only \emph{Step 2} of Lemma~\ref{golab} needs to be adapted. By assumption, $v_n$ converges uniformly to $v$. Therefore there exists a rank $N'>N$ such that 
\[\|v_n-v\|_{\infty} \leqslant \eta.\]
We consider $n> N'$. Since $u_n$ is $\alpha$-Hölder and bounded from above by 1, we get for all $y$,
\begin{align*}
    |v_n(x) - v_n(y)| &= |u_n(x)^2 + \delta_\eps - u_n(y)^2 - \delta_\eps|,\\
    &\leqslant |u_n(x)+u_n(y)||u_n(x)-u_n(y)|, \\
    &\leqslant 2C|x-y|^\alpha.
\end{align*}
Thus, up to considering $\eta$ and $r$ small enough (with respect to $v(x)$) we can guarantee the positivity of the term 
\begin{align*}
    &v_n(x)-2Cr^\alpha(1+C\eta)^\alpha \geqslant v(x)-2Cr^\alpha(1+C\eta)^\alpha -\eta \geqslant 0.
\end{align*}
Hence, 
\begin{align*}
    &\mu_n(\overline{B}(x,r(1+C\eta))) = \int_{S_{\ell_n}\cap \overline{B}(x,r(1+C\eta)) } v_n(y) \dm \Haus(y), \\
    &\geqslant \int_{S_{\ell_n}\cap \overline{B}(x,r(1+C\eta)) } (v_n(x)-2Cr^\alpha(1+C\eta)^\alpha) \dm \Haus(y), \\
    &= (v_n(x)-2Cr^\alpha(1+C\eta)^\alpha)\Haus(S_{\ell_n}\cap \overline{B}(x,r(1+C\eta))), \\
    &\geqslant (v_n(x)-2Cr^\alpha(1+C\eta)^\alpha)\pi r^2(1-C\eta)^2,\\
    &\geqslant (v(x)-2Cr^\alpha(1+C\eta)^\alpha -\eta)\pi r^2(1-C\eta)^2.
\end{align*}
Then, letting $n \to \infty$ yields
\begin{align*}
    \mu(\overline{B}(x,r(1+C\eta))&\geqslant \limsup_n \mu_n(\overline{B}(x,r(1+C\eta)),\\
    &\geqslant (v(x)-2Cr^\alpha(1+C\eta)^\alpha - \eta) \pi r^2(1-C\eta)^2.
\end{align*} 
Finally, we let $r\to 0$, yielding
$$\lim_{r\to 0} \frac{\mu(\overline{B}(x,r(1+C\eta)))}{\pi r^2 (1+C\eta)^2}\geqslant (v(x)-\eta) \frac{(1-C\eta)^2}{(1+C\eta)^2} ,$$
and we conclude by letting $\eta \to 0$ 
$$f(x) = \lim_{r\to 0} \frac{\mu(\overline{B}(x,r))}{\pi r^2}\geqslant v(x),$$
which achieves the proof of the lower semi-continuity.
\end{proof}


\bibliography{biblio2}

\begin{thebibliography}{1}

\bibitem{adams2003sobolev}
Robert~A Adams and John~JF Fournier.
\newblock {\em Sobolev spaces}, volume 140.
\newblock Elsevier, 2003.

\bibitem{AFP00}
Luigi Ambrosio, Nicola Fusco, and Diego Pallara.
\newblock {\em Functions of bounded variation and free discontinuity problems}.
\newblock Oxford Math. Monogr. Oxford: Clarendon Press, 2000.

\bibitem{bonnivard2025phasefieldapproximationplateaus}
Matthieu Bonnivard, Elie Bretin, Antoine Lemenant, and Eve Machefert.
\newblock Phase field approximation for plateau's problem: a curve geodesic
  distance penalty approach, 2025.

\bibitem{bonnivard2018phase}
Matthieu Bonnivard, Antoine Lemenant, and Vincent Millot.
\newblock On a phase field approximation of the planar steiner problem:
  existence, regularity, and asymptotic of minimizers.
\newblock {\em Interfaces and free Boundaries}, 20(1):69--106, 2018.

\bibitem{bonnivard2015approximation}
Matthieu Bonnivard, Antoine Lemenant, and Filippo Santambrogio.
\newblock Approximation of length minimization problems among compact connected
  sets.
\newblock {\em SIAM Journal on Mathematical Analysis}, 47(2):1489--1529, 2015.

\bibitem{gilbargelliptic}
David Gilbarg, Neil~S Trudinger, David Gilbarg, and NS~Trudinger.
\newblock {\em Elliptic partial differential equations of second order}, volume
  224.
\newblock Springer, 1977.

\bibitem{Zi}
William~P. Ziemer.
\newblock {\em Weakly differentiable functions. {Sobolev} spaces and functions
  of bounded variation}, volume 120 of {\em Grad. Texts Math.}
\newblock Berlin etc.: Springer-Verlag, 1989.

\end{thebibliography}
\bibliographystyle{plain}

 \end{document}